\documentclass[12pt]{amsart}
\usepackage{amsmath}
\usepackage{amstext}
\usepackage{amssymb}
\usepackage{amsthm}
\swapnumbers
\theoremstyle{plain}
\newtheorem{thm}{Theorem}[section]
\newtheorem{lem}[thm]{Lemma}

\newtheorem{prop}[thm]{Proposition}
\newtheorem{cor}[thm]{Corollary}

\theoremstyle{definition}
\newtheorem{defi}[thm]{Definition}

\newtheorem{qu}[thm]{Question}

\newtheorem{rmds}[thm]{Reminders}
\theoremstyle{remark}
\newtheorem*{note}{Note}

\newtheorem{rmks}[thm]{Remarks}

 \DeclareMathOperator{\Id}{Id}

\DeclareMathOperator{\Hom}{Hom}

 \DeclareMathOperator{\Spec}{Spec}
 
 \DeclareMathOperator{\ann}{ann}
\DeclareMathOperator{\grann}{gr-ann} 

\def\N{\mathbb N}

\def\fa{{\mathfrak{a}}}
\def\fA{{\mathfrak{A}}}
\def\fb{{\mathfrak{b}}}
\def\fB{{\mathfrak{B}}}
\def\fc{{\mathfrak{c}}}
\def\fC{{\mathfrak{C}}}
\def\fd{{\mathfrak{d}}}

\def\fg{{\mathfrak{g}}}
\def\fh{{\mathfrak{h}}}

\def\fm{{\mathfrak{m}}}
\def\fM{{\mathfrak{M}}}

\def\fp{{\mathfrak{p}}}
\def\fP{{\mathfrak{P}}}
\def\fq{{\mathfrak{q}}}
\def\fQ{{\mathfrak{Q}}}

\def\nn{\relax\ifmmode{\mathbb N_{0}}\else$\mathbb N_{0}$\fi}
\def\lra{\longrightarrow}

\begin{document}

\vspace{0.5in}

\title{Graded annihilators and uniformly $F$-compatible ideals}
\author{RODNEY Y. SHARP}
\address{{\it School of Mathematics and Statistics\\
University of Sheffield\\ Hicks Building\\ Sheffield S3 7RH\\ United
Kingdom}} \email{R.Y.Sharp@sheffield.ac.uk}


\date{\today}


\begin{abstract} Let $R$ be a commutative (Noetherian) local ring of
prime characteristic $p$ that is $F$-pure. This paper is concerned
with comparison of three finite sets of radical ideals of $R$, one
of which is only defined in the case when $R$ is $F$-finite (that
is, is finitely generated when viewed as a module over itself via
the Frobenius homomorphism). Two of the afore-mentioned three sets
have links to tight closure, via test ideals. Among the aims of the
paper are a proof that two of the sets are equal, and a proposal for
a generalization of I. M. Aberbach's and F. Enescu's splitting
prime. \vspace{0.2in}

\end{abstract}

\maketitle

\setcounter{section}{-1}
\section{\it Introduction}
\label{intro}

Throughout the paper, let $(R,\fm)$ be a commutative (Noetherian)
local ring of prime characteristic $p$ having maximal ideal $\fm$.
In recent years, the study of $R$-modules with a Frobenius action
has assisted in the development of the theory of tight closure over
$R$. An $R$-module with a Frobenius action can be viewed as a left
module over the Frobenius skew polynomial ring over $R$, and such
left modules will play a central r\^ole in this paper.

The Frobenius skew polynomial ring over $R$ is described as follows.
Throughout, $f : R \lra R$ denotes the Frobenius ring homomorphism,
for which $f(r) = r^p$ for all $r \in R$. The {\em Frobenius skew
polynomial ring over $R$} is the skew polynomial ring $R[x,f]$
associated to $R$ and $f$ in the indeterminate $x$; as a left
$R$-module, $R[x,f]$ is freely generated by $(x^i)_{i \geq 0}$, and
so consists of all polynomials $\sum_{i = 0}^n r_i x^i$, where $n
\geq 0$ and $r_0,\ldots,r_n \in R$; however, its multiplication is
subject to the rule $xr = f(r)x = r^px$ for all $r \in R.$

We can think of $R[x,f]$ as a positively-graded ring $R[x,f] =
\bigoplus_{n=0}^{\infty} R[x,f]_n$, where $R[x,f]_n = Rx^n$ for $n
\geq 0$. The {\em graded annihilator\/} of a left $R[x,f]$-module
$H$ is the largest graded two-sided ideal of $R[x,f]$ that
annihilates $H$; it is denoted by $\grann_{R[x,f]}H$.

Let $G$ be a left $R[x,f]$-module that is $x$-torsion-free in the
sense that $xg = 0$, for $g \in G$, only when $g = 0$. Then
$\grann_{R[x,f]}G = \fb R[x,f]$, where $\fb = (0:_RG)$ is a radical
ideal. See \cite[Lemma 1.9]{ga}. We shall use ${\mathcal I}(G)$ (or
${\mathcal I}_R(G)$) to denote the set of $R$-annihilators of the
$R[x,f]$-submodules of $G$; we shall refer to the members of
${\mathcal I}(G)$ as the {\em $G$-special $R$-ideals\/}. For a
graded two-sided ideal $\fB$ of $R[x,f]$, we denote by $\ann_G(\fB)$
or $\ann_G\fB$ the $R[x,f]$-submodule of $G$ consisting of all
elements of $G$ that are annihilated by $\fB$. Also, we shall use
${\mathcal A}(G)$ to denote the set of {\em special annihilator
submodules of $G$}, that is, the set of $R[x,f]$-submodules of $G$
of the form $\ann_G(\fA)$, where $\fA$ is a graded two-sided ideal
of $R[x,f]$. In \cite[\S 1]{ga}, the present author showed that
there is a sort of `Galois' correspondence between ${\mathcal I}(G)$
and ${\mathcal A}(G)$. In more detail, there is an order-reversing
bijection, $\Delta : \mathcal{A}(G) \lra \mathcal{I}(G)$ given by \[
\Delta : N \longmapsto \left(\grann_{R[x,f]}N\right)\cap R =
(0:_RN). \] The inverse bijection, $\Delta^{-1} : \mathcal{I}(G)
\lra \mathcal{A}(G),$ also order-reversing, is given by \[
\Delta^{-1} : \fb \longmapsto \ann_G\left(\fb R[x,f])\right). \]

We shall be mainly concerned in this paper with the situation where
$R$ is $F$-pure. We remind the reader what this means. For $j \in
\N$ (the set of positive integers) and an $R$-module $M$, let
$M^{(j)}$ denote $M$ considered as a left $R$-module in the natural
way and as a right $R$-module via $f^j$, the $j$th iterate of the
Frobenius ring homomorphism. Then $R$ is {\em $F$-pure\/} if, for
every $R$-module $M$, the natural map $M \lra R^{(1)} \otimes_R M$
(which maps $m \in M$ to $1 \otimes m$) is injective.

Note that $R^{(j)} \cong Rx^j$ as $(R,R)$-bimodules. Let $i\in\nn$,
the set of non-negative integers. When we endow $Rx^i$ and $Rx^j$
with their natural structures as $(R,R)$-bimodules (inherited from
their being graded components of $R[x,f]$), there is an isomorphism
of (left) $R$-modules $\phi : Rx^{i+j}\otimes_R M
\stackrel{\cong}{\lra} Rx^{i}\otimes_R(Rx^{j}\otimes_R M)$ for which
$\phi(rx^{i+j} \otimes m) = rx^i \otimes (x^j \otimes m)$ for all $r
\in R$ and $m \in M$. It follows that $R$ is $F$-pure if and only if
the left $R[x,f]$-module $R[x,f]\otimes_RM$ is $x$-torsion-free for
every $R$-module $M$.  This means that, when $R$ is $F$-pure, there
is a good supply of natural $x$-torsion-free left $R[x,f]$-modules.

In fact, we shall use $\Phi$ (or $\Phi_R$ when it is desirable to
specify which ring is being considered) to denote the functor
$R[x,f]\otimes_R \;{\scriptscriptstyle \bullet}\;$ from the category
of $R$-modules (and all $R$-homomorphisms) to the category of all
$\nn$-graded left $R[x,f]$-modules (and all homogeneous
$R[x,f]$-homomorphisms). For an $R$-module $M$, we shall identify
$\Phi(M)$ with $\bigoplus_{n\in\nn}Rx^n\otimes_RM$, and (usually)
identify its $0$th component $R\otimes_RM$ with $M$, in the obvious
ways.

Let $E$ be the injective envelope of the simple $R$-module $R/\fm$.
We shall be concerned with $\Phi(E)$, the $\nn$-graded left
$R[x,f]$-module $\bigoplus_{n\in\nn}Rx^n\otimes_RE$. Assume now that
$R$ is $F$-pure. In \cite[Corollary 4.11]{Fpurhastest}, the present
author proved that the set ${\mathcal I}(\Phi(E))$ is a finite set
of radical ideals of $R$; in \cite[Theorem 3.6 and Corollary
3.7]{ga}, he proved that ${\mathcal I}(\Phi(E))$ is closed under
taking primary (prime in this case) components; and in
\cite[Corollary 2.8]{hbf}, he proved that the big test ideal
$\widetilde{\tau}(R)$ of $R$ (for tight closure) is equal to the
smallest member of ${\mathcal I}(\Phi(E))$ that meets $R^{\circ}$,
the complement in $R$ of the union of the minimal prime ideals of
$R$.

Let $\fa \in {\mathcal I}(\Phi(E))$ (with $\fa \neq R$), still in
the $F$-pure case. The special annihilator submodule
$\ann_{\Phi(E)}(\fa R[x,f])$ of $\Phi(E)$ corresponding to $\fa$
inherits a natural structure as a graded left module over the
Frobenius skew polynomial ring $(R/\fa)[x,f]$, and its $0$th
component is contained in $(0:_E\fa)$. As $R/\fa$-module, the latter
is isomorphic to the injective envelope of the simple
$R/\fa$-module. Motivated by results in \cite[\S 3]{hbf} in the case
where $R$ is complete, and by work of K. Schwede in \cite[\S
5]{Schwe10} in the $F$-finite case, we say that $\fa$ is {\em fully
$\Phi(E)$-special\/} if (it is $\Phi(E)$-special and) its $0$th
component is exactly $(0:_E\fa)$. The main result of this paper is
that a $\Phi(E)$-special ideal of $R$ is always fully
$\Phi(E)$-special provided that $R$ is an ($F$-pure) homomorphic
image of an excellent regular local ring of characteristic $p$. When
$R$ satisfies this condition, corollaries can be drawn from that
main result: we shall establish an analogue of \cite[Theorem
3.1]{hbf} and, in particular, show that $R/\fa$ is $F$-pure whenever
$\fa$ is a proper $\Phi(E)$-special ideal of $R$.

Along the way, we shall show that, in the case where $R$ is
$F$-finite as well as $F$-pure, the set ${\mathcal I}(\Phi(E))$ of
$\Phi(E)$-special ideals of $R$ is equal to the set of {\em
uniformly $F$-compatible ideals\/} of $R$, introduced by K. Schwede
in \cite[\S 3]{Schwe10}. An ideal $\fb$ of $R$ is said to be {\em
uniformly $F$-compatible\/} if, for every $j > 0$ and every $\phi
\in \Hom_R(R^{(j)}, R)$, we have $\phi(\fb^{(j)}) \subseteq \fb$. In
\cite[Corollary 5.3 and Corollary 3.3]{Schwe10}, Schwede proved that
there are only finitely many uniformly $F$-compatible ideals of $R$
and that they are all radical; in \cite[Proposition 4.7 and
Corollary 4.8]{Schwe10}, he proved that the set of uniformly
$F$-compatible ideals is closed under taking primary (prime in this
case) components; in \cite[Theorem 6.3]{Schwe10}, Schwede proved
that the big test ideal $\widetilde{\tau}(R)$ of $R$ is equal to the
smallest uniformly $F$-compatible ideal of $R$ that meets
$R^{\circ}$; and in \cite[Remark 4.4 and Proposition 4.7]{Schwe10},
he proved that there is a unique largest proper uniformly
$F$-compatible ideal of $R$, and that that is prime and equal to the
splitting prime of $R$ discovered and defined by I. M. Aberbach and
F. Enescu \cite[\S3]{AbeEne05}.

Thus, in the $F$-finite $F$-pure case, the set of uniformly
$F$-compatible ideals of $R$ has properties similar to some
properties of ${\mathcal I}(\Phi(E))$. Are the two sets the same? We
shall, during the course of the paper, show that the answer is
'yes'.  It should be emphasized, however, that Schwede only defined
uniformly $F$-compatible ideals in the $F$-finite case, whereas the
majority of this paper is devoted to the study of fully
$\Phi(E)$-special ideals in the ($F$-pure but) not necessarily
$F$-finite case.

We shall use the notation of this Introduction throughout the
remainder of the paper. In particular, $R$ will denote a local ring
of prime characteristic $p$ having maximal ideal $\fm$. We shall
sometimes use the notation $(R,\fm)$ just to remind the reader that
$R$ is local. The completion of $R$ will be denoted by ${\widehat
R}$. We shall only assume that $R$ is reduced, or $F$-pure, or
$F$-finite, when there is an explicit statement to that effect; also
$E$ will continue to denote $E_R(R/\fm)$. We continue to use $\N$,
respectively $\nn$, to denote the set of all positive, respectively
non-negative, integers.

For $j \in \nn$, the $j$th component of an $\nn$-graded left
$R[x,f]$-module $G$ will be denoted by $G_j$.

\section{Fully $\Phi(E)$-special ideals}
\label{fsi}

We remind the reader that we usually identify the $0$th component of
$\Phi(E) = \bigoplus_{n\in\nn}Rx^n\otimes_RE$ with $E$ in the
obvious natural way. For an ideal $\fa$ of $R$, we have, with this
convention, that the $0$th component of $\ann_{\Phi(E)}(\fa R[x,f])$
is contained in $(0:_E\fa)$.

\begin{lem}
\label{fsi.1} Assume that $(R,\fm)$ is $F$-pure; let $\fa$ be an
ideal of $R$. Then the $0$th component $(\ann_{\Phi(E)}(\fa
R[x,f]))_0$ of $\ann_{\Phi(E)}(\fa R[x,f])$ contains $(0:_E\fa)$ if
and only if $\fa$ is $\Phi(E)$-special and $(\ann_{\Phi(E)}(\fa
R[x,f]))_0 = (0:_E\fa)$.
\end{lem}

\begin{proof} Only the implication `$\Rightarrow$' needs proof.

Assume that $(0:_E\fa) \subseteq (\ann_{\Phi(E)}(\fa R[x,f]))_0$.
Since $\ann_{\Phi(E)}(\fa R[x,f])$ is an $R[x,f]$-submodule of
$\Phi(E)$, it follows that $\ann_{\Phi(E)}(\fa R[x,f])$ contains the
image $J$ of the map
$$\Phi((0:_E\fa)) = R[x,f]\otimes_R(0:_E\fa) \lra R[x,f]\otimes_RE =
\Phi(E)$$ induced by inclusion. Let $\fb$ be the radical ideal of
$R$ for which $\grann_{R[x,f]}J = \fb R[x,f]$, so that $\fb =
(0:_RJ)$. As $J \subseteq \ann_{\Phi(E)}(\fa R[x,f])$, we must have
$\fa \subseteq \fb$. Furthermore, $\fb$ annihilates $(0:_E\fa) \cong
\Hom_R(R/\fa,E)$, and since an $R$-module and its Matlis dual have
the same annihilator, we also have $\fb \subseteq \fa$. Thus $\fa =
\fb$ is the $R$-annihilator of an $R[x,f]$-submodule of $\Phi(E)$,
and so $\fa \in {\mathcal I}(\Phi(E))$.

Finally, note that an $e \in (\ann_{\Phi(E)}(\fa R[x,f]))_0$ must be
annihilated by $\fa$, and so lies in $(0:_E\fa)$.
\end{proof}

\begin{defi}
\label{fsi.2} Assume that $(R,\fm)$ is $F$-pure; let $\fa$ be an
ideal of $R$. We say that $\fa$ is {\em fully $\Phi(E)$-special\/}
if the equivalent conditions of Lemma \ref{fsi.1} are satisfied.

Thus $\fa$ is fully $\Phi(E)$-special if and only if $(0:_E\fa)
\subseteq (\ann_{\Phi(E)}(\fa R[x,f]))_0$, and, then, $\fa$ is
$\Phi(E)$-special and we have the equality $ (0:_E\fa) =
(\ann_{\Phi(E)}(\fa R[x,f]))_0. $
\end{defi}

To facilitate the presentation of some examples of $\Phi(E)$-special
ideals that are fully $\Phi(E)$-special, we review next the theory
of $S$-tight closure, where $S$ is a multiplicatively closed subset
of $R$.  This theory was developed in \cite{hbf}. The special case
of the theory in which $S = R^{\circ}$ is the `classical' tight
closure theory of M. Hochster and C. Huneke \cite{HocHun90}.

\begin{rmds}
\label{hbfsect1} Let $H$ be a left $R[x,f]$-module and let $S$ be a
multiplicatively closed subset of $R$.
\begin{enumerate}
\item We define the {\em internal $S$-tight closure of zero in $H$,\/}
denoted $\Delta^S(H)$, to be the $R[x,f]$-submodule of $H$ given by
$$
\Delta^S(H)= \left\{ h \in H : \text{there exists~} s \in S
\text{~with~} sx^nh = 0 \text{~for all~} n \gg 0 \right\}.
$$
When $M$ is an $R$-module and we take the graded left
$R[x,f]$-module $\Phi(M) = R[x,f]\otimes_RM$ for $H$, the
$R[x,f]$-submodule $\Delta^S(\Phi(M))$ of $\Phi(M)$ is graded, and
we refer to its $0$th component as the {\em $S$-tight closure of\/
$0$ in $M$}, or the {\em tight closure with respect to $S$ of\/ $0$
in $M$}, and denote it by $0^{*,S}_M$. See \cite[\S 1]{hbf}.

\item By \cite[Example 1.3(ii)]{hbf}, we have, for an $R$-module
$M$,
$$\Delta^S(R[x,f]\otimes_RM) = 0^{*,S}_M \oplus 0^{*,S}_{Rx
\otimes_RM} \oplus \cdots \oplus 0^{*,S}_{Rx^n \otimes_RM}\oplus
\cdots.
$$

\item Recall that an {\em $S$-test element for $R$} is an element $s
\in S$ such that, for every $R$-module $M$ and every $j \in \nn$,
the element $sx^j$ annihilates $1 \otimes m \in (\Phi(M))_0$ for
every $m \in 0^{*,S}_M$. The ideal of $R$ generated by all the
$S$-test elements for $R$ is called the {\em $S$-test ideal of $R$},
and denoted by $\tau^S(R)$.
\end{enumerate}
\end{rmds}

\begin{rmds}
\label{hbf1.5} Suppose that $(R,\fm)$ is $F$-pure. Let $S$ be a
multiplicatively closed subset of $R$. Recall that the set
${\mathcal I}(\Phi(E))$ of $\Phi(E)$-special $R$-ideals is finite;
let $\fb^{S,\Phi(E)}$ denote the intersection of all the minimal
members of the set $$\left\{ \fp \in \Spec(R) \cap
\mathcal{I}(\Phi(E)) : \fp \cap S \neq \emptyset\right\}.$$ Thus
$\fb^{S,\Phi(E)}$ is the smallest member of $\mathcal{I}(\Phi(E))$
that meets $S$.

\begin{enumerate}
\item By \cite[Theorem 2.6]{hbf}, the set $S \cap
\fb^{S,\Phi(E)}$ is (non-empty and) equal to the set of $S$-test
elements for $R$.
\item Thus there exists an $S$-test element
for $R$.
\item Furthermore, $\Delta^S(\Phi(E)) =
\ann_{\Phi(E)}(\fb^{S,\Phi(E)}R[x,f])$ and $(0:_R\Delta^S(\Phi(E)))
= \fb^{S,\Phi(E)}$, by \cite[Proposition 1.5]{hbf}.
\item By \cite[Proposition 2.10(v)]{hbf}, we have $\fb^{S,\Phi(E)} =
(0:_R0^{*,S}_E)$.
\end{enumerate}
\end{rmds}

\begin{lem}
[Sharp {\cite[Corollary 2.8]{hbf}}]  \label{hbf1.6}

Suppose that $(R,\fm)$ is $F$-pure. Let $S$ be the complement in $R$
of the union of finitely many prime ideals.

Then the $S$-test ideal $\tau^S(R)$ is equal to $\fb^{S,\Phi(E)}$,
the smallest member of the finite set ${\mathcal I}(\Phi(E))$ that
meets $S$.
\end{lem}

We shall also use the following result from \cite{hbf}.

\begin{thm}
[Sharp {\cite[Theorem 2.12]{hbf}}] \label{hbf2.12} Suppose that
$(R,\fm)$ is $F$-pure. Let $\fa \in \mathcal{I}(\Phi(E))$. Then
there exists a multiplicatively closed subset $S$ of $R$ such that
$\fa$ is the $S$-test ideal of $R$. Moreover, $S$ can be taken to be
the complement in $R$ of the union of finitely many prime ideals.
\end{thm}

We are now able to give examples of fully $\Phi(E)$-special ideals,
because the next result shows that, when $(R,\fm)$ is complete and
$F$-pure, a $\Phi(E)$-special ideal of $R$ is automatically fully
$\Phi(E)$-special.

\begin{prop}
\label{res.2} Suppose that $(R,\fm)$ is complete and $F$-pure. Then
every $\Phi(E)$-special ideal of $R$ is fully $\Phi(E)$-special.
\end{prop}

\begin{proof} Let $\fa$ be a $\Phi(E)$-special ideal of $R$. If $\fa = R$,
then $$(0:_E\fa) = 0 \subseteq \ann_{\Phi(E)}(\fa R[x,f])$$ and
$\fa$ is fully $\Phi(E)$-special. We therefore assume that $\fa$ is
proper.

By Theorem \ref{hbf2.12} and \cite[Corollary 2.8]{hbf}, there exist
finitely many prime ideals $\fp_1, \ldots, \fp_n$ of $R$ such that,
if we set $S := R\setminus \bigcup_{i=1}^n\fp_i$, then $\fa$ is the
$S$-test ideal of $R$, that is $ \fa = \tau^S(R) = \fb^{S,\Phi(E)},$
where the notation is as in \ref{hbfsect1}(iii) and \ref{hbf1.5}.
Therefore, by \ref{hbfsect1}(ii) and \ref{hbf1.5}(iii),
\begin{align*}
0^{*,S}_E \oplus 0^{*,S}_{Rx \otimes_RE} \oplus \cdots \oplus
0^{*,S}_{Rx^n \otimes_RE}\oplus \cdots & = \Delta^S(\Phi(E))\\  & =
\ann_{\Phi(E)}(\fb^{S,\Phi(E)}R[x,f]).
\end{align*}
Now we know that $\fb^{S,\Phi(E)} = (0:_R0^{*,S}_E)$, by
\ref{hbf1.5}(iv). Since $R$ is complete, it follows from Matlis
duality (see, for example, \cite[p.\ 154]{SV}) that $0^{*,S}_E =
(0:_E\fb^{S,\Phi(E)})$. We have thus shown that $(0:_E\fa) = R
\otimes_R(0:_E\fa) \subseteq (\ann_{\Phi(E)}(\fa R[x,f]))_0$. Thus
$\fa$ is fully $\Phi(E)$-special.
\end{proof}

Next, we develop some theory for fully $\Phi(E)$-special ideals.

\begin{lem}
\label{res.5z} Suppose that $(R,\fm)$ is $F$-pure, and let $\fa$ be
a fully $\Phi(E)$-special ideal of $R$. Then $\fa$ is radical and
every associated prime of $\fa$ is also fully $\Phi(E)$-special.
\end{lem}

\begin{proof} We can assume that $\fa$ is proper. Since $\fa$ is
$\Phi(E)$-special, it must be radical. Let $\fa = \fp_1 \cap \cdots
\cap\fp_t$ be the minimal primary (prime in this case) decomposition
of $\fa$, and let $i \in \{1, \ldots, t\}$.

Since $\fa$ is fully $\Phi(E)$-special, we have $(0:_{E}\fa)
\subseteq (\ann_{\Phi(E)}(\fa R[x,f]))_0$. Let $e \in (0:_E\fp_i)$
and let $r \in \fp_i$. We show that $rx^n$ annihilates the element
$1 \otimes e$ of the $0$th component of $\Phi(E)$. There exists
$$
a \in \bigcap^t_{\stackrel{\scriptstyle j=1}{j\neq i}} \fp_j
\setminus \fp_i.
$$
Now $(0:_E\fp_i) = a(0:_E\fp_i)$, because multiplication by $a$
provides a monomorphism of $R/\fp_i$ into itself and $E$ is
injective. Therefore $e = ae'$ for some $e' \in (0:_E\fp_i)$.
Therefore $rx^n\otimes e = rx^n\otimes ae' = ra^{p^n}x^n \otimes e'
= 0$ since $ra^{p^n} \in \fa$ and
$$(0:_E\fp_i)\subseteq (0:_{E}\fa) \subseteq
\ann_{\Phi(E)}(\fa R[x,f]).$$ Therefore $(0:_E\fp_i)\subseteq
(\ann_{\Phi(E)}(\fp_i R[x,f]))_0$ and $\fp_i$ is fully
$\Phi(E)$-special.
\end{proof}

\begin{prop}
\label{res.5} Suppose that $(R,\fm)$ is $F$-pure. Let
$(\fa_{\lambda})_{\lambda \in \Lambda}$ be a non-empty family of
fully $\Phi(E)$-special ideals of $R$. Then $\sum_{\lambda \in
\Lambda}\fa_{\lambda}$ is again fully $\Phi(E)$-special.
\end{prop}

\begin{proof} Set $\fa := \sum_{\lambda \in
\Lambda}\fa_{\lambda}$, and observe that $\fa R[x,f] = \sum_{\lambda
\in \Lambda}(\fa_{\lambda} R[x,f])$. By assumption, we have
$(0:_E\fa_{\lambda}) \subseteq \ann_{\Phi(E)}(\fa_{\lambda} R[x,f])$
for all $\lambda \in \Lambda$. It follows that
\begin{align*}
(0:_E\fa) & = {\textstyle \left(0:_E\sum_{\lambda \in
\Lambda}\fa_{\lambda}\right) = \bigcap_{\lambda \in
\Lambda}\left(0:_E\fa_{\lambda}\right)}\\ & \subseteq {\textstyle
\bigcap_{\lambda \in \Lambda}(\ann_{\Phi(E)}(\fa_{\lambda} R[x,f])})_0\\
& = {\textstyle \left(\ann_{\Phi(E)}\left(\sum_{\lambda \in
\Lambda}\left(\fa_{\lambda} R[x,f]\right)\right)\right)_0 =
(\ann_{\Phi(E)}(\fa R[x,f]))_0}.
\end{align*}
Therefore $\fa := \sum_{\lambda \in \Lambda}\fa_{\lambda}$ is fully
$\Phi(E)$-special.
\end{proof}

\begin{cor}
\label{res.6} Suppose that $(R,\fm)$ is $F$-pure. Then $R$ has a
unique largest fully $\Phi(E)$-special proper ideal, and this is
prime.
\end{cor}

\begin{proof} The zero ideal is fully $\Phi(E)$-special, and so
it follows from Proposition \ref{res.5} that the sum $\fb$ of all
the fully $\Phi(E)$-special proper ideals of $R$ is fully
$\Phi(E)$-special (and contained in $\fm$), and so is the unique
largest fully $\Phi(E)$-special proper ideal of $R$. Also $\fb$ must
be prime, since all the associated primes of $\fb$ are fully
$\Phi(E)$-special, by Lemma \ref{res.5z}.
\end{proof}

In what follows, we shall have cause to pass between $R$ and its
completion. Note that if $R$ is $F$-pure, then so too is
$\widehat{R}$, by Hochster and Roberts \cite[Corollary
6.13]{HocRob74}. The following technical lemma will be helpful.

\begin{lem}
\label{lm.3} {\rm (See \cite[Lemma 4.3]{btctefsnrer}.)} There is a
unique way of extending the $R$-module structure on $E :=
E_R(R/\fm)$ to an ${\widehat R}$-module structure. Recall that, as
an $\widehat{R}$-module, $E \cong
E_{\widehat{R}}(\widehat{R}/\widehat{\fm})$.

Since each element of $\Phi_R(E) = R[x,f]\otimes_RE$ is annihilated
by some power of $\fm$, the left $R[x,f]$-module structure on
$\Phi_R(E)$ can be extended in a unique way to a left ${\widehat
R}[x,f]$-module structure.

The map $\beta : \Phi_R(E) = R[x,f]\otimes_RE \lra
\widehat{R}[x,f]\otimes_{\widehat{R}}E = \Phi_{\widehat{R}}(E)$ for
which
$$\beta(rx^i \otimes h) = rx^i \otimes h \quad \text{for all~} r \in R,\ i \in
\nn \text{~and~} h \in E$$ is a homogeneous
$\widehat{R}[x,f]$-isomorphism.

Since each element of $\Phi_{R}(E)$ is annihilated by some power of
$\fm$, it follows that a subset of $\Phi_{R}(E)$ is an
$R[x,f]$-submodule if and only if it is an
$\widehat{R}[x,f]$-submodule. Consequently, $${\mathcal
I}_R(\Phi_R(E)) = \left\{ \fB \cap R : \fB \in {\mathcal
I}_{\widehat{R}}(\Phi_{\widehat{R}}(E)) \right\}.$$
\end{lem}

\begin{lem}
\label{fsi.3} Suppose that $(R,\fm)$ is $F$-pure, and let $\fa$ be
an ideal of $R$. Then $\fa \widehat{R}$ is a fully
$\Phi_{\widehat{R}}(E)$-special ideal of $\widehat{R}$ if and only
if $\fa$ is a fully $\Phi_R(E)$-special ideal of $R$.
\end{lem}

\begin{proof} By Lemma \ref{lm.3}, when we extend the left $R[x,f]$-module structure on
$\Phi_R(E)$, in the unique way possible, to a left ${\widehat
R}[x,f]$-module structure, $E \cong
E_{\widehat{R}}(\widehat{R}/\widehat{\fm})$ as $\widehat R$-modules
and $\Phi_R(E) \cong \Phi_{\widehat{R}}(E)$ as left ${\widehat
R}[x,f]$-modules. The claim therefore follows from the facts that
$$\ann_{\Phi_R(E)}(\fa R[x,f]) = \ann_{\Phi_R(E)}((\fa \widehat
R)\widehat R[x,f])$$ and $(0:_E \fa) =(0:_E \fa\widehat R)$.
\end{proof}

\section{\it The case where $R$ is an $F$-pure homomorphic image of
an excellent regular local ring of characteristic $p$}

The main aim of this section is to prove that, when $R$ is an
$F$-pure homomorphic image of an excellent regular local ring of
characteristic $p$, every $\Phi(E)$-special ideal of $R$ is fully
$\Phi(E)$-special ideal. This will enable us to extend some results
obtained in \cite[\S 3]{hbf} about an $F$-pure complete local ring
to an $F$-pure homomorphic image of an excellent regular local ring
of characteristic $p$.  We begin the section with a lemma that is
derived from a result of G. Lyubeznik \cite[Lemma 4.1]{Lyube97}.

\begin{lem}
\label{pa.1} Let $(S,\fM)$ be a complete regular local ring of
characteristic $p$, and let $\fB$ be a proper, non-zero ideal of
$S$. Denote $E_S(S/\fM)$ by $E_S$, and let $S[x,f]$ denote the
Frobenius skew polynomial ring over $S$. Let $n \in \N$.

Since $S$ is regular, $S^{(n)}$ is faithfully flat over $S$, and we
identify $Sx^n\otimes_S(0:_{E_S}\fB)$ as an $S$-submodule of
$Sx^n\otimes_SE_S$ in the natural way. Let $a_1, \ldots, a_d$ be a
regular system of parameters for $S$. Consider the $S$-isomorphism
$\delta_n: Sx^n \otimes_S E_S \stackrel{\cong}{\lra} E_S$ of {\rm
\cite[4.2(iii)]{ga}}, for which (with the notation used in the
statement of that result)
$$
\delta_n \left( bx^n \otimes \left[\frac{s}{(a_1\ldots
a_d)^j}\right] \right) = \left[\frac{bs^{p^n}}{(a_1\ldots
a_d)^{jp^n}}\right] \quad \mbox{~for all~} b,s \in S \mbox{~and~} j
\in \nn.
$$
The isomorphism $\delta_n$ maps
\begin{enumerate}
\item $Sx^n\otimes_S(0:_{E_S}\fB)$
onto $(0:_{E_S}\fB^{[p^n]})$, and
\item  $\fB(Sx^n\otimes_S(0:_{E_S}\fB))$
onto $(0:_{E_S}(\fB^{[p^n]}:\fB))$.
\end{enumerate}
\end{lem}

\begin{proof} (i) Use of the analogue of Lyubeznik \cite[Lemma
4.1]{Lyube97} for the functor $Sx^n\otimes_S\: {\scriptscriptstyle
\bullet} \:$ shows that the Matlis dual of
$Sx^n\otimes_S(0:_{E_S}\fB)$ is $S$-isomorphic to $Sx^n \otimes_S
(S/\fB) \cong S/\fB^{[p^n]}$. Since each $S$-module has the same
annihilator as its Matlis dual, we thus see that
$Sx^n\otimes_S(0:_{E_S}\fB)$ has annihilator $\fB^{[p^n]}$. Since
$S$ is complete, we have $T = (0:_{E_S}(0:_ST))$ for each submodule
$T$ of $E_S$, by Matlis duality (see, for example, \cite[p.\
154]{SV}). It therefore follows that
$$
\delta_n(Sx^n\otimes_S(0:_{E_S}\fB)) =(0:_{E_S}\fB^{[p^n]}).
$$
(ii) Set $N := Sx^n\otimes_S(0:_{E_S}\fB)$. Similar reasoning shows
that
$$
\delta_n(\fB N) = (0:_{E_S}(0:_S\fB N)) = (0:_{E_S}((0:_SN) : \fB))
= (0:_{E_S}(\fB^{[p^n]}:\fB)).
$$
\end{proof}

\begin{prop}
\label{pa.2} Suppose that $R = S/\fA$, where $(S,\fM)$ is a regular
local ring of characteristic $p$, and $\fA$ is a proper ideal of
$S$. Assume also that $R$ is $F$-pure. Let $\fb$ be a proper ideal
of $R$; let $\fB$ be the unique ideal of $S$ that contains $\fA$ and
is such that $\fB/\fA = \fb$.

Then $\fb$ is fully $\Phi(E)$-special if and only if
$(\fA^{[p^n]}:\fA) \subseteq (\fB^{[p^n]}:\fB)$ for all $n \in \N$.
\end{prop}

\begin{note} In the $F$-finite case, this result is already known
and due to K. Schwede \cite[Proposition 3.11 and Lemma
5.1]{Schwe10}.
\end{note}

\begin{proof} If $\fA = 0$, then $R$ is regular, so that its big test ideal
is $R$ itself (by \cite[Theorem 8.8]{LyuSmi01}, for example) and the
only proper $\Phi(E)$-special ideal of $R$ is $0$; also,
$(0^{[p^n]}:0) = S$, and the only proper ideal $\fB$ of $S$
satisfying $(0^{[p^n]}:0) \subseteq (\fB^{[p^n]}:\fB)$ for all
$n\in\N$ is the zero ideal. Thus the result is true when $\fA = 0$;
we therefore assume for the remainder of this proof that $\fA \neq
0$.

Note that $\widehat{R} = \widehat{S}/\fA\widehat{S}$ is again
$F$-pure and that $\widehat{S}$ is an excellent complete regular
local ring of characteristic $p$, with maximal ideal $\fM
\widehat{S}$.

We also note that $\fb$ is a fully $\Phi_R(E)$-special ideal of $R$
if and only if $\fb \widehat{R}$ is a fully
$\Phi_{\widehat{R}}(E)$-special ideal of $\widehat{R}$, by Lemma
\ref{fsi.3}. Furthermore, by the faithful flatness of $\widehat{S}$
over $S$, we have, for $n \in \N$,
$$
((\fA \widehat{S})^{[p^n]} : \fA \widehat{S}) = (\fA^{[p^n]} :
\fA)\widehat{S} \subseteq (\fB^{[p^n]} : \fB)\widehat{S} = (\fB
\widehat{S})^{[p^n]} : \fB \widehat{S})
$$
if and only if $(\fA^{[p^n]} : \fA) \subseteq (\fB^{[p^n]} : \fB)$.
Therefore, we can, and do, assume henceforth in this proof that $S$
is complete.

Let $E_S := E_S(S/\fM)$. Now $(0:_{E_S}\fA) = E := E_R(R/\fm)$ and
$(0:_{E_S}\fB) = (0:_{E}\fb)$. Note that $\fb$ is fully
$\Phi_R(E)$-special if and only if, for each $n \in \N$ and each $r
\in \fb$, the element $rx^n \in Rx^n$ annihilates the $R$-submodule
$(0:_{E}\fb)$ of the $0$th component $E$ of $\Phi_R(E)$.

Let $n \in \N$. There is an exact sequence of $(S,S)$-bimodules
$$
0 \lra \fA Sx^n\stackrel{\subseteq}{\lra} Sx^n \stackrel{\nu}{\lra}
Rx^n \lra 0,
$$ where $\nu(sx^n) = (s + \fA)x^n$ for all $s \in S$. The map $$Sx^n
\otimes_S(0:_{E_S}\fA) \lra Rx^n \otimes_S(0:_{E_S}\fA) = Rx^n
\otimes_R(0:_{E_S}\fA) = Rx^n \otimes_RE$$ induced by $\nu$
therefore has kernel $\fA(Sx^n \otimes_S(0:_{E_S}\fA))$.

It follows that $\fb$ is fully $\Phi_R(E)$-special if and only if,
for all $n \in \N$, $s\in \fB$ and $g \in (0:_{E_S}\fB) =
(0:_{E}\fb)$, the element $sx^n \otimes g$ of $Sx^n
\otimes_S(0:_{E_S}\fA)$ lies in $$\fA(Sx^n
\otimes_S(0:_{E_S}\fA)).$$ In other words, $\fb$ is fully
$\Phi_R(E)$-special if and only if, for all $n \in \N$, we have
$$\fB(Sx^n \otimes_S(0:_{E_S}\fB)) \subseteq \fA(Sx^n
\otimes_S(0:_{E_S}\fA)).$$ (We are here identifying $Sx^n
\otimes_S(0:_{E_S}\fB)$ and $Sx^n \otimes_S(0:_{E_S}\fA)$ with
submodules of $Sx^n \otimes_SE_S$ in the obvious ways, using the
faithful flatness of $S^{(n)}$ over $S$.)

By \cite[4.2(iii)]{ga}, we have $Sx^n \otimes_SE_S \cong E_S$. Since
$S$ is complete, each submodule $T$ of $E_S$ satisfies $T =
(0:_{E_S}(0:_ST))$. Set $N := Sx^n \otimes_SE_S$. Thus
$$
\fA(Sx^n \otimes_S(0:_{E_S}\fA)) = (0:_{N}(0:_S(\fA(Sx^n
\otimes_S(0:_{E_S}\fA))))) = (0:_{N}(\fA^{[p^n]}:\fA)),
$$
by Lemma \ref{pa.1}. Similarly, $\fB(Sx^n \otimes_S(0:_{E_S}\fB)) =
(0:_{N}(\fB^{[p^n]}:\fB))$. It follows that $\fb$ is fully
$\Phi_R(E)$-special if and only if
$$
(0:_{N}(\fB^{[p^n]}:\fB)) \subseteq (0:_{N}(\fA^{[p^n]}:\fA)) \quad
\mbox{for all~} n \in \N,
$$
that is (since $N \cong E_S$), if and only if $(\fA^{[p^n]}:\fA)
\subseteq (\fB^{[p^n]}:\fB)$ for all $n \in \N$.
\end{proof}

\begin{thm}
\label{pa.3} Suppose that $R = S/\fA$ is a homomorphic image of an
excellent regular local ring $(S,\fM)$ of characteristic $p$, modulo
a proper ideal $\fA$. Assume that $R$ is $F$-pure.

Then each $\Phi(E)$-special ideal of $R$ is fully $\Phi(E)$-special.
\end{thm}

\begin{proof} Once again, the claim is easy to prove if $\fA = 0$,
and so we assume henceforth in this proof that $\fA \neq 0$

Note that $\widehat{R} = \widehat{S}/\fA\widehat{S}$ is again
$F$-pure and that $\widehat{S}$ is an excellent complete regular
local ring of characteristic $p$, with maximal ideal $\fM
\widehat{S}$.

Let $\fb$ be a $\Phi(E)$-special $R$-ideal with $\fb \neq R$. Then
$\fb = \fc\cap R$ for some $\Phi_{\widehat{R}}(E)$-special
${\widehat R}$-ideal $\fc$. (We have used Lemma \ref{lm.3} here.)
Let $\fC$ be the unique ideal of ${\widehat S}$ that contains $\fA
{\widehat S}$ and is such that $\fC/\fA {\widehat S} = \fc$. By
Proposition \ref{res.2}, the ideal $\fc$ of ${\widehat R}$ is fully
$\Phi_{\widehat{R}}(E)$-special, and so, by Proposition \ref{pa.2},
we have
$$
(\fA^{[p^n]}:\fA){\widehat S} = ((\fA{\widehat
S})^{[p^n]}:\fA{\widehat S}) \subseteq (\fC^{[p^n]}:\fC) \quad
\mbox{for all~} n \in \N.
$$
Set $\fC \cap S := \fB$, so that $\fB/\fA = \fb$.

Let $n \in \N$ and $s \in (\fA^{[p^n]}:_S\fA)$. Therefore $s \in
(\fC^{[p^n]}:\fC)$. It follows from G. Lyubeznik and K. E. Smith
\cite[Lemma 6.6]{LyuSmi01} that $\fC^{[p^n]}\cap S = (\fC\cap
S)^{[p^n]}$. (Lyubeznik's and Smith's proof of this result uses work
of N. Radu \cite[Corollary 5]{Radu92}, which, in turn, uses D.
Popescu's general N\'eron desingularization \cite{Popes85,
Popes86}.) We can now deduce that
$$
s(\fC\cap S) \subseteq s\fC \cap S \subseteq \fC^{[p^n]}\cap S =
(\fC\cap S)^{[p^n]},
$$
so that $s \in ((\fC\cap S)^{[p^n]}:\fC\cap S) = (\fB^{[p^n]}:\fB)$.

We have thus shown that $(\fA^{[p^n]}:\fA) \subseteq
(\fB^{[p^n]}:\fB)$ for all $n\in \N$, so that $\fb = \fB/\fA$ is
fully $\Phi(E)$-special by Proposition \ref{pa.2}.
\end{proof}

In the case where $R$ is an $F$-pure homomorphic image of an
excellent regular local ring of characteristic $p$, the
characterization of ${\mathcal I}(\Phi(E))$ afforded by Proposition
\ref{pa.2} and Theorem \ref{pa.3} enables us to see that that set
behaves well under localization. As the ideals in ${\mathcal
I}(\Phi(E))$ are precisely those that can be expressed as
intersections of finitely many prime members of ${\mathcal
I}(\Phi(E))$, it is of interest to examine the behaviour of
${\mathcal I}(\Phi(E))\cap \Spec (R)$ under localization. The next
proposition, which is an extension of part of \cite[Proposition
2.8]{Fpurhastest}, is in preparation for this investigation.

\begin{prop} \label{pa.4} Let $S$ be a regular local ring of
characteristic $p$, and let $n \in \N$. Let $\fA,\fB_1, \ldots,
\fB_t,\fC$ be ideals of $S$ with $0 \neq \fA \neq S$, and let $\fA =
\fQ_1 \cap \ldots \cap \fQ_t$ be a minimal primary decomposition of
$\fA$.
\begin{enumerate}
\item We have $(\fB_1 \cap \cdots \cap \fB_t)^{[p^n]} = \fB_1 ^{[p^n]}\cap \cdots \cap
\fB_t^{[p^n]}$.
\item If $\fQ$ is a $\fP$-primary ideal of $S$, then $\fQ^{[p^n]}$ is also
$\fP$-primary.
\item The equation $\fA^{[p^n]} = \fQ_1^{[p^n]} \cap \cdots \cap
\fQ_t^{[p^n]}$ provides a minimal primary decomposition of
$\fA^{[p^n]}$.
\item We have $(\fA :\fC)^{[p^n]} = (\fA ^{[p^n]}:\fC^{[p^n]})$ and
$(\fA^{[p^n]} : \fA) \subseteq \left((\fA :\fC)^{[p^n]} : (\fA :
\fC)\right)$.
\item If $\fP$ is an associated prime ideal of $\fA$, then $(\fA^{[p^n]} : \fA)
\subseteq (\fP^{[p^n]} : \fP)$.
\item Since $0 \neq \fA \neq S$, we have $(\fA^{[p^n]} : \fA) \neq S$.
If $\fP_1 := \sqrt{\fQ_1}$ is a minimal prime ideal of $\fA$, then
$\fP_1$ is a minimal prime ideal of $(\fA^{[p^n]} : \fA)$ and the
unique $\fP_1$-primary component of $(\fA^{[p^n]} : \fA)$ is
$(\fQ_1^{[p^n]} : \fQ_1)$.
\end{enumerate}
\end{prop}

\begin{proof} Parts (i), (ii) and (iii) were essentially proved in
\cite[Proposition 2.8]{Fpurhastest}, while parts (iv), (v) and (vi)
can be proved by obvious modifications of the arguments used to
prove the corresponding parts of \cite[Proposition
2.8]{Fpurhastest}.
\end{proof}

\begin{cor}
\label{pa.5} Suppose that $R$ is $F$-pure and a homomorphic image of
an excellent regular local ring $S$ of characteristic $p$ modulo a
proper ideal $\fA$. Let $\fp \in \Spec (R)$. Then
$$
{\mathcal I}_{R_{\fp}}(\Phi_{R_{\fp}}(E_{R_{\fp}}(R_{\fp}/\fp
R_{\fp}))) \cap \Spec (R_{\fp}) = \left\{ \fq R_{\fp} : \fq \in
{\mathcal I}(\Phi(E))\cap \Spec (R) \mbox{~and~} \fq \subseteq
\fp\right\}.
$$
\end{cor}

\begin{proof} Note that, by M. Hochster and J. L. Roberts
\cite[Lemma 6.2]{HocRob74}, the localization $R_{\fp}$ is again
$F$-pure. The claim is easy to prove when $\fA = 0$, and so we
assume that $\fA \neq 0$.

For each lower case fraktur letter that denotes an ideal of $R$, let
the corresponding upper case fraktur letter denote the unique ideal
of $S$ that contains $\fA$ and has quotient modulo $\fA$ equal to
the specified ideal of $R$. For example, $\fP$ denotes the unique
ideal of $S$ that contains $\fA$ and is such that $\fP/\fA = \fp$.

Note that $R_{\fp} \cong S_{\fP}/\fA S_{\fP}$ is again a homomorphic
image of an excellent regular local ring $S$ of characteristic $p$.
Let $\fq \in \Spec (R)$ with $\fq \subseteq \fp$.

Suppose first that $\fq \in {\mathcal I}(\Phi(E))\cap \Spec (R)$. By
Theorem \ref{pa.3}, we see that $\fq$ is fully $\Phi(E)$-special;
use of Proposition \ref{pa.2} shows that $(\fA^{[p^n]}:\fA)
\subseteq (\fQ^{[p^n]}:\fQ)$ for all $n \in \N$. Therefore
$$
((\fA S_{\fP})^{[p^n]}:\fA S_{\fP}) \subseteq ((\fQ
S_{\fP})^{[p^n]}:\fQ S_{\fP}) \quad \mbox{for all~} n \in \N.
$$
Since the standard isomorphism $S_{\fP}/\fA S_{\fP}
\stackrel{\cong}{\lra} R_{\fp}$ maps $\fQ S_{\fP}/\fA S_{\fP}$ onto
$\fq R_{\fp}$, it follows from Proposition \ref{pa.2} that $\fq
R_{\fp}$ is fully $\Phi_{R_{\fp}}(E_{R_{\fp}}(R_{\fp}/\fp
R_{\fp}))$-special.

Conversely, suppose that $\fq R_{\fp}$ is
$\Phi_{R_{\fp}}(E_{R_{\fp}}(R_{\fp}/\fp R_{\fp}))$-special, so that,
by Theorem \ref{pa.3}, it is fully
$\Phi_{R_{\fp}}(E_{R_{\fp}}(R_{\fp}/\fp R_{\fp}))$-special. By
Proposition \ref{pa.2}, this means that
$$
((\fA S_{\fP})^{[p^n]}:\fA S_{\fP}) \subseteq ((\fQ
S_{\fP})^{[p^n]}:\fQ S_{\fP}) \quad \mbox{for all~} n \in \N.
$$
Let $^e$ and $^c$ denote extension and contraction of ideals under
the natural ring homomorphism $S \lra S_{\fP}$. Contract the last
displayed inclusion relations back to $S$ to see that
$$
(\fA ^{[p^n]}:\fA ) \subseteq (\fA ^{[p^n]}:\fA )^{ec} \subseteq
(\fQ ^{[p^n]}:\fQ )^{ec} = (\fQ ^{[p^n]}:\fQ ) \quad \mbox{for all~}
n \in \N
$$
because $(\fQ ^{[p^n]}:\fQ)$ is $\fQ$-primary (for all $n\in\N$), by
Proposition \ref{pa.4}(vi). It follows from Proposition \ref{pa.2}
that $\fQ/\fA = \fq$ is fully $\Phi(E)$-special.
\end{proof}

We can now recover a special case of a result of Lyubeznik and
Smith.

\begin{cor} [G. Lyubeznik and K. E. Smith {\cite[Theorem 7.1]{LyuSmi01}}]
\label{pa.6} Suppose that $R$ is $F$-pure and a homomorphic image of
an excellent regular local ring $S$ of characteristic $p$ modulo a
proper ideal $\fA$. Let $\fp \in \Spec (R)$. Then the big test ideal
of $R_{\fp}$ is the extension to $R_{\fp}$ of the big test ideal of
$R$. In symbols, $\widetilde{\tau}(R_{\fp}) =
\widetilde{\tau}(R)R_{\fp}$.
\end{cor}

\begin{proof} The big test ideal $\widetilde{\tau}(R)$ of $R$ is
equal to the intersection of the (finitely many) members of
${\mathcal I}(\Phi(E))\cap \Spec (R)$ of positive height, and a
similar statement holds for $R_{\fp}$. The claim therefore follows
from Corollary \ref{pa.5}.
\end{proof}

Some results were obtained in \cite[Theorem 3.1]{hbf} for an
$F$-pure complete local ring of characteristic $p$. We can now use
Theorem \ref{pa.3} to establish analogous results for an $F$-pure
homomorphic image of an excellent regular local ring of
characteristic $p$.

\begin{thm}
\label{fp.11} Suppose $(R,\fm)$ is $F$-pure and that every
$\Phi(E)$-special ideal of $R$ is fully $\Phi(E)$-special. (For
example, by Theorem {\rm \ref{pa.3}}, this would be the case if $R$
were a homomorphic image of an excellent regular local ring of
characteristic $p$.) Let $\fc$ be a proper ideal of $R$ that is
$\Phi(E)$-special. In the light of Theorem\/ {\rm \ref{hbf2.12}},
let $\fp_1, \ldots,\fp_w$ be prime ideals of $R$ for which the
multiplicatively closed subset $S = R \setminus
\bigcup_{i=1}^w\fp_i$ of $R$ satisfies $\fc = \tau^{S}(R)$. Set $ J
:= \Delta^{S}(\Phi(E)),$ a graded left $R[x,f]$-module.
\begin{enumerate}
\item We have
$ J = 0^{*,S}_{E} \oplus 0^{*,S}_{Rx\otimes_RE} \oplus \cdots \oplus
0^{*,S}_{Rx^n\otimes_RE} \oplus \cdots = \ann_{\Phi(E)}(\fc R[x,f]).
$
\item When we regard $J$ as a graded left
$(R/\fc)[x,f]$-module in the natural way, it is $x$-torsion-free and
has ${\mathcal I}_{R/\fc}(J) = \left\{\fg/\fc : \fg \in {\mathcal
I}(\Phi(E)) : \fg \supseteq \fc \right\}.$
\item The $0$th component $J_0$ of $J$ is $(0:_{E}\fc)$; as
$R/\fc$-module, this is isomorphic to
$E_{R/\fc}((R/\fc)/(\fm/\fc))$.
\item The ring $R/\fc$ is $F$-pure.
\item We have ${\mathcal I}(\Phi_{R/\fc}(J_0)) \subseteq {\mathcal
I}_{R/\fc}(J)$, so that
$$
\left\{ \fd : \fd \mbox{~is an ideal of $R$ with~} \fd \supseteq \fc
\mbox{~and~} \fd/\fc \in {\mathcal I}(\Phi_{R/\fc}(J_0))\right\}
\subseteq {\mathcal I}(\Phi_R(E)).
$$
\end{enumerate}
\end{thm}

\begin{proof} Since the
$\Phi(E)$-special ideal $\fc$ is fully $\Phi(E)$-special, we have
$J_0 = (0:_{E}\fc)$. Given this observation, one can now use the
arguments employed in the proof of \cite[Theorem 3.1]{hbf} to
furnish a proof of this theorem.
\end{proof}

The next corollary follows from Theorem \ref{fp.11} just as, in
\cite{hbf}, Corollary 3.2 follows from Theorem 3.1.

\begin{cor}
\label{fp.12} Suppose that $(R,\fm)$ is local, $F$-pure and that
every $\Phi(E)$-special ideal of $R$ is fully $\Phi(E)$-special.
(For example, by Theorem {\rm \ref{pa.3}}, this would be the case if
$R$ were a homomorphic image of an excellent regular local ring of
characteristic $p$.) Let $\fc$ be a proper ideal of $R$ that is
$\Phi(E)$-special. Denote $R/\fc$ by $\overline{R}$, and note that
$\overline{R}$ is $F$-pure, by Theorem\/ {\rm \ref{fp.11}(iv)}. Let
$T$ be a multiplicatively closed subset of $\overline{R}$ which is
the complement in $\overline{R}$ of the union of finitely many prime
ideals. The {\em finitistic $T$-test ideal $\tau^{{\rm fg},T}
(\overline{R})$ of $\overline{R}$\/} is defined to be
$\bigcap_L(0:_{\overline{R}}0^{*,T}_L)$, where the intersection is
taken over all finitely generated $\overline{R}$-modules $L$.

\begin{enumerate}
\item If\/ $\fh$ denotes the unique ideal of $R$ that contains $\fc$ and
is such that $\fh/\fc = \tau^{{\rm fg},T}(\overline{R})$, the
finitistic $T$-test ideal of $\overline{R}$, then $\fh \in
\mathcal{I}(\Phi(E))$.
\item In particular, if\/ $\fh'$ denotes the unique ideal of $R$ that contains $\fc$ and
is such that $\fh'/\fc = \tau(\overline{R})$, the test ideal of
$\overline{R}$, then $\fh' \in \mathcal{I}(\Phi(E))$.
\item If\/ $\fg$ denotes the unique ideal of $R$ that contains $\fc$ and
is such that $\fg/\fc = \tau^T(\overline{R})$, the $T$-test ideal of
$\overline{R}$, then $\fg \in \mathcal{I}(\Phi(E))$.
\item In particular, if\/ $\fg'$ denotes the unique ideal of $R$ that contains $\fc$ and
is such that $\fg'/\fc = \widetilde{\tau}(\overline{R})$, the big
test ideal of $\overline{R}$, then $\fg' \in \mathcal{I}(\Phi(E))$.
\end{enumerate}
\end{cor}

\begin{proof} Straightforward modifications of the arguments given
in the proof of \cite[Corollary 3.2]{hbf} will provide a proof for
this.
\end{proof}

\begin{lem}
\label{hbffp.16} Assume that $(R,\fm)$ is local, $F$-pure and a
homomorphic image of an excellent regular local ring of
characteristic $p$.
\begin{enumerate}
\item There is a strictly ascending chain $0
= \tau_0 \subset \tau_1 \subset \cdots \subset \tau_t \subset
\tau_{t+1} = R$ of radical ideals of $R$ such that, for each $i = 0,
\ldots, t$, the reduced local ring $R/\tau_i$ is $F$-pure and its
test ideal is $\tau_{i+1}/\tau_i$. We call this the {\em test ideal
chain of $R$}. All of $\tau_0 = 0,\tau_1, \cdots ,\tau_t$, and all
their associated primes, belong to ${\mathcal I}(\Phi(E))$.
\item There is a strictly ascending chain $0
= \widetilde{\tau}_0 \subset \widetilde{\tau}_1 \subset \cdots
\subset \widetilde{\tau}_w \subset \widetilde{\tau}_{w+1} = R$ of
radical ideals in $\mathcal{I}(\Phi(E))$ such that, for each $i = 0,
\ldots, w$, the reduced local ring $R/\widetilde{\tau}_i$ is
$F$-pure and its big test ideal is
$\widetilde{\tau}_{i+1}/\widetilde{\tau}_i$. We call this the {\em
big test ideal chain of $R$}. All of $\widetilde{\tau}_0 =
0,\widetilde{\tau}_1, \cdots ,\widetilde{\tau}_w$, and all their
associated primes, belong to ${\mathcal I}(\Phi(E))$.
\end{enumerate}
\end{lem}

\begin{note} In the case when $R$ is an ($F$-pure) homomorphic image
of an $F$-finite regular local ring, part (i) of this result is
known and due to Janet Cowden Vassilev \cite[\S3]{Cowde98}.
\end{note}

\begin{proof} (i) Set $\tau_1 := \tau(R)$, and note that $\tau(R)
\in {\mathcal I}(\Phi(E))$. If $\tau_1 \neq R$, apply Theorem
\ref{fp.11} with the choice $\fc = \tau(R) = \tau_1$. That shows
that $R/\tau_1$ is $F$-pure. Now argue by induction on $\dim R$,
noting that $R/\tau_1$ is a homomorphic image of an excellent
regular local ring of characteristic $p$. Use Theorem \ref{fp.11}(v)
to show that all of $\tau_0, \tau_1, \ldots, \tau_t$ belong to
${\mathcal I}(\Phi(E))$.

(ii) This is proved similarly.
\end{proof}

\section{\it The $F$-finite case}

In the $F$-finite case, the results above have strong connections
with work of K. Schwede in \cite{Schwe10}, and the purpose of this
section is to explore some of those connections. The introduction
contains a description of certain properties of the set of all
uniformly $F$-compatible ideals in an $F$-finite, $F$-pure local
ring $R$, and some of these are similar to properties of the set of
all fully $\Phi(E)$-special ideals of $R$: we shall show in this
section that, in this special case, an ideal of $R$ is uniformly
$F$-compatible if and only if it is $\Phi(E)$-special, and that this
is the case if and only if it is fully $\Phi(E)$-special.

\begin{defi}\label{nd.1} Suppose that $R$ is $F$-finite; let $\fb$
be an ideal of $R$. Then $\fb$ is said to be {\em uniformly
$F$-compatible\/} if, for every $n > 0$ and every $\phi \in
\Hom_R(R^{(n)}, R)$, we have $\phi(\fb^{(n)}) \subseteq \fb$.
\end{defi}

\begin{prop}
[Schwede {\cite[Lemma 5.1]{Schwe10}}]  \label{Schw5.1} Suppose that
$(R,\fm)$ is $F$-finite; let $\fb$ be an ideal of $R$. Then $\fb$ is
uniformly $F$-compatible if and only if $(0:_E\fb) \subseteq
(\ann_{\Phi(E)}(\fb R[x,f]))_0$.

Thus when $R$ is $F$-finite and $F$-pure, $\fb$ is uniformly
$F$-compatible if and only if it is fully $\Phi(E)$-special.
\end{prop}

\begin{proof} Let $n\in \N$ and $r \in R$. Multiplication
by $r$ yields an $R$-homomorphism of $R^{(n)}$, which, strictly
speaking, we should denote by $r\Id_{R^{(n)}}$. Also $f^n : R \lra
R^{(n)}$ is an $R$-homomorphism. Thus we can consider the
composition of $R$-homomorphisms $ R \stackrel{f^n}{\lra} R^{(n)}
\stackrel{r}{\longrightarrow} R^{(n)} .$

Application of the functor $\: {\scriptscriptstyle \bullet} \:
\otimes_RE$ yields a composition of $R$-homomorphisms $$R \otimes_RE
\lra R^{(n)}\otimes_RE \stackrel{r}{\longrightarrow}
R^{(n)}\otimes_RE,$$ where the `$r$' over the second arrow is an
abbreviation for $r\Id_{R^{(n)}}\otimes_R E$. But $R^{(n)} \cong
Rx^n$ as $(R,R)$-bimodules; furthermore, $(0:_E\fb) \cong
\Hom_R(R/\fb,E)$. It follows that $(0:_E\fb) \subseteq
(\ann_{\Phi(E)}(\fb R[x,f]))_0$ if and only if, for all $n \in \N$
and all $r \in \fb$, the composition
$$
(0:_E\fb) \stackrel{\subseteq}{\lra}E \stackrel{\cong}{\lra} R
\otimes_RE \lra R^{(n)}\otimes_RE \stackrel{r}{\longrightarrow}
R^{(n)}\otimes_RE
$$
(in which the second map is the natural isomorphism) is zero.

Let $M$ be an $R$-module. Recall that there is an $R$-homomorphism
\[
\xi_{M} : M \otimes_R E \longrightarrow \Hom_R(\Hom_R(M,R),E)
\]
such that, for $m \in M$, $e \in E$ and $g \in \Hom_R(M,R)$, we have
$\left(\xi_{M}(m \otimes e) \right) (g) = g(m)e$\@. Furthermore, as
$M$ varies, the $\xi_{M}$ constitute a natural transformation of
functors; also $\xi_{M}$ is an isomorphism whenever $M$ is finitely
generated. We shall use $D$ to denote the functor $\Hom_R(\:
{\scriptscriptstyle \bullet} \: ,E)$.

Since $R^{(n)}$ is a finitely generated $R$-module, $(0:_E\fb)
\subseteq (\ann_{\Phi(E)}(\fb R[x,f]))_0$ if and only if, for all $n
\in \N$ and all $r \in \fb$, the composition
$$
D(R/\fb) \rightarrow D(R) \stackrel{\cong}{\rightarrow}
D(\Hom_R(R,R)) \rightarrow D(\Hom_R(R^{(n)}, R))
\stackrel{r}{\rightarrow} D(\Hom_R(R^{(n)}, R))
$$
is zero. (Here, the first map is induced from the natural
epimorphism $R \lra R/\fb$, the second map is the natural
isomorphism, and the sequence from the middle term rightwards is the
result of application of the functor $\Hom_R(\Hom_R(\:
{\scriptscriptstyle \bullet} \:,R),E)$ to the composition $ R
\stackrel{f^n}{\lra} R^{(n)} \stackrel{r}{\longrightarrow} R^{(n)}$
described at the beginning of the proof.)

Since $D$ is a faithful functor (because $E$ is an injective
cogenerator for $R$), we can deduce that $(0:_E\fb) \subseteq
(\ann_{\Phi(E)}(\fb R[x,f]))_0$ if and only if, for all $n \in \N$
and all $r \in \fb$, the composition
$$
\Hom_R(R^{(n)}, R)\stackrel{r}{\longrightarrow}\Hom_R(R^{(n)}, R)
\lra \Hom_R(R,R) \stackrel{\cong}{\lra} R \lra R/\fb
$$
is zero, that is, if and only if $\fb$ is uniformly $F$-compatible.
\end{proof}

\begin{prop} [Schwede {\cite{Schwe10}}] \label{Schw} Suppose that
$(R,\fm)$ is $F$-finite, and let $\fa$ be an ideal of $R$. Note that
the completion ${\widehat R}$ of $R$ is again $F$-finite.

\begin{enumerate}
\item If $\fa$ is a uniformly
$F$-compatible ideal of $R$, then $\fa{\widehat R}$ is a uniformly
$F$-compatible ideal of ${\widehat R}$. See Schwede \cite[Lemma
3.9]{Schwe10}.
\item If $\fC$ is a uniformly
$F$-compatible ideal of ${\widehat R}$, then $\fC \cap R$ is a
uniformly $F$-compatible ideal of $R$. See Schwede \cite[Lemma
3.8]{Schwe10}.
\end{enumerate}
\end{prop}

\begin{proof} For a finitely generated $R$-module $M$, we identify
${\widehat M}$ with $M \otimes_R{\widehat R}$ in the usual way, and
we note that there is a natural ${\widehat R}$-isomorphism $\psi_M :
\Hom_R(M,R)\otimes_R{\widehat R} \stackrel{\cong}{\lra}
\Hom_{{\widehat R}}(M\otimes_R{\widehat R},R\otimes _R{\widehat R})$
for which $\psi_M(g \otimes {\widehat r}) = {\widehat r}(g \otimes
\Id_{{\widehat R}})$ for all $g \in \Hom_R(M,R)$ and ${\widehat r}
\in {\widehat R}$.  Let $n \in \N$. Consideration of Cauchy
sequences shows that $\widehat{M^{(n)}} = {\widehat M}^{(n)}$. In
particular, $\widehat{R^{(n)}} = {\widehat R}^{(n)}$ and
$\widehat{\fa^{(n)}} = (\widehat{\fa})^{(n)} = (\fa{\widehat
R})^{(n)}$.

There is an ${\widehat R}$-isomorphism $\gamma :
R^{(n)}\otimes_R{\widehat R}  \stackrel{\cong}{\lra} {\widehat
R}^{(n)}$ which maps $\fa^{(n)}\otimes_R{\widehat R}$ onto
$(\fa{\widehat R})^{(n)}$. Also, the natural ${\widehat
R}$-isomorphism $\delta : R\otimes_R{\widehat
R}\stackrel{\cong}{\lra} {\widehat R}$ maps $\fa \otimes_R{\widehat
R}$ onto $\fa {\widehat R}$.

(i) Let $\theta \in \Hom_{{\widehat R}}(R^{(n)}\otimes_R{\widehat
R},R\otimes _R{\widehat R})$. By the above, there exist $\phi_1,
\ldots, \phi_t \in \Hom_R(R^{(n)},R)$ and ${\widehat r}_1, \ldots,
{\widehat r}_t \in {\widehat R}$ such that $\theta = {\widehat
r}_1(\phi_1\otimes\Id_{{\widehat R}}) + \cdots + {\widehat
r}_t(\phi_t\otimes\Id_{{\widehat R}})$. Since $\phi_i(\fa^{(n)})
\subseteq \fa$ for all $n \in \N$ and $i = 1, \ldots, t$, we see
that $\theta(\fa^{(n)}\otimes_R{\widehat R}) \subseteq
\fa\otimes_R{\widehat R}$ for all $n\in\N$. Use of the
above-mentioned isomorphisms $\gamma$ and $\delta$ now enables us to
conclude that $\fa{\widehat R}$ is a uniformly $F$-compatible ideal
of ${\widehat R}$.

(ii) Let $\phi \in \Hom_R(R^{(n)},R)$, and set $\fc := \fC \cap R$.
Then $$\phi\otimes\Id_{{\widehat R}} \in \Hom_{{\widehat
R}}(R^{(n)}\otimes_R{\widehat R},R\otimes _R{\widehat R})$$ and
$\delta \circ (\phi\otimes\Id_{{\widehat R}}) \circ \gamma^{-1}$
maps $\fC^{(n)}$ into $\fC$, and therefore maps $(\fc {\widehat
R})^{(n)}$ into $\fC$. Therefore $\delta \circ
(\phi\otimes\Id_{{\widehat R}})$ maps $\fc^{(n)}\otimes _R{\widehat
R}$ into $\fC$, so that $\phi(a) \in \fC\cap R = \fc$ for all $a \in
\fc^{(n)}$. Therefore $\fc$ is a uniformly $F$-compatible ideal of
$R$.
\end{proof}

\begin{thm}
\label{res.3} Suppose that $(R,\fm)$ is $F$-pure and $F$-finite.
Then each $\Phi(E)$-special ideal $\fa$ of $R$ is automatically
fully $\Phi(E)$-special.
\end{thm}

\begin{proof} Note that ${\widehat R}$ is also $F$-pure, by Hochster
and Roberts \cite[Corollary 6.13]{HocRob74}. Also, ${\widehat R}$ is
$F$-finite, because the completion of the finitely generated
$R$-module $R^{(1)}$ is ${\widehat R}^{(1)}$.

Thus, by definition, $\fa$ is the $R$-annihilator of an
$R[x,f]$-submodule of $\Phi(E)$. It follows from Lemma \ref{lm.3}
that $\fa = \fA \cap R$ for some ideal $\fA$ of $\widehat{R}$ that
is the $\widehat{R}$-annihilator of an $\widehat{R}[x,f]$-submodule
of $\Phi_{\widehat{R}}(E)$. Thus $\fA$ is
$\Phi_{\widehat{R}}(E)$-special. It follows from Proposition
\ref{res.2} that $\fA$ is a fully $\Phi_{\widehat{R}}(E)$-special
ideal of $\widehat{R}$, and so is uniformly $F$-compatible, by
Proposition \ref{Schw5.1}. Therefore, by Proposition \ref{Schw}(ii),
the contraction $\fA \cap R = \fa$ is a uniformly $F$-compatible
ideal of $R$, and is therefore fully $\Phi(E)$-special, by
Proposition \ref{Schw5.1} again.
\end{proof}

\begin{cor}
\label{res.3c} Suppose that $(R,\fm)$ is $F$-pure and $F$-finite;
let $\fa$ be an ideal of $R$. Then the following statements are
equivalent:
\begin{enumerate}
\item $\fa$ is uniformly $F$-compatible;
\item $\fa$ is $\Phi(E)$-special;
\item $\fa$ is fully $\Phi(E)$-special.
\end{enumerate}
\end{cor}

\begin{proof} This is now immediate from Proposition \ref{Schw5.1} and Theorem
\ref{res.3}.
\end{proof}

\begin{qu}
\label{res.4} Suppose that $(R,\fm)$ is $F$-pure.

We have seen that each $\Phi(E)$-special ideal of $R$ is fully
$\Phi(E)$-special if $R$ is complete (by Proposition \ref{res.2}) or
if $R$ is a homomorphic image of an excellent regular local ring of
characteristic $p$ (by Theorem \ref{pa.3}) or if $R$ is $F$-finite
(by Theorem \ref{res.3}).

Note that each complete local ring is excellent, and that each
$F$-finite local ring of characteristic $p$ is excellent (by E. Kunz
\cite[Theorem 2.5]{Kunz76}). The above results raise the following
question.  If the $F$-pure local ring $R$ is excellent, is it the
case that every $\Phi(E)$-special ideal of $R$ is fully
$\Phi(E)$-special?
\end{qu}

\section{\it A generalization of Aberbach's and Enescu's splitting
prime}

Recall from \cite[Remark 2.8 and Proposition 2.9]{LyuSmi01} that G.
Lyubeznik and K. E. Smith defined $(R,\fm)$ to be {\em strongly
$F$-regular\/} (even in the case where $R$ is not $F$-finite)
precisely when the zero submodule of $E$ is tightly closed in $E$.
See M. Hochster and C. Huneke \cite[\S 8]{HocHun90}.

\begin{thm}
\label{res.6p} Suppose that $(R,\fm)$ is $F$-pure and that every
$\Phi(E)$-special ideal of $R$ is fully $\Phi(E)$-special. (For
example, by Theorem {\rm \ref{pa.3}}, this would be the case if $R$
were a homomorphic image of an excellent regular local ring of
characteristic $p$; it would also be the case if $R$ were
$F$-finite, by Theorem {\rm \ref{res.3}}.)
\begin{enumerate}
\item There exists a unique largest $\Phi(E)$-special proper ideal, $\fc$ say, of $R$ and
this is prime. Furthermore, $R/\fc$ is strongly $F$-regular.

\item Let $T$ be the $R[x,f]$-submodule of $\Phi(E)$ generated by
$(0:_E\fm) \subseteq R\otimes_RE$. Then $\grann_{R[x,f]}T = \fc
R[x,f]$.
\end{enumerate}
\end{thm}

\begin{proof} (i) By Corollary \ref{res.6}, there is a
unique largest $\Phi(E)$-special proper ideal $\fc$ of $R$, and this
is prime. By Corollary \ref{fp.12}(iv), the big test ideal of
$R/\fc$ is $R/\fc$ itself, so that $1_{R/\fc}$ is a big test element
for $R/\fc$. Therefore the zero submodule of $E_{R/\fc}(R/\fm)$ is
tightly closed in $E_{R/\fc}(R/\fm)$, and so $R/\fc$ is strongly
$F$-regular.

(ii) Note that $T$ is the image of the $R[x,f]$-homomorphism
$$
R[x,f] \otimes_R (0:_E\fm) \lra R[x,f] \otimes_R E = \Phi(E)
$$
induced by the inclusion map $(0:_E\fm) \stackrel{\subseteq}{\lra}
E$. Let $\fd$ be the $\Phi(E)$-special ideal of $R$ for which
$\grann_{R[x,f]}T = \fd R[x,f]$. Since $\fd$ annihilates
$(0:_E\fm)$, we see that $\fd$ is proper. Suppose that there exists
$\fh \in {\mathcal I}(\Phi(E))$ such that $\fd \subset \fh \subseteq
\fm$. (The symbol `$\subset$' is reserved to denote strict
inclusion.) Thus we have $(0:_E\fm) \subseteq (0:_E\fh) \subseteq
(0:_E\fd)$. But we know that every $\Phi(E)$-special ideal of $R$ is
fully $\Phi(E)$-special, and therefore $(0:_E\fh) \subseteq
(\ann_{\Phi(E)}(\fh R[x,f]))_0$. Since $\ann_{\Phi(E)}(\fh R[x,f])$
is an $R[x,f]$-submodule of $\Phi(E)$, it follows that
$$
T \subseteq \ann_{\Phi(E)}(\fh R[x,f]) \subseteq \ann_{\Phi(E)}(\fd
R[x,f]).
$$
Now take graded annihilators: in view of the bijective
correspondence between the sets ${\mathcal I}(\Phi(E))$ and
${\mathcal A}(\Phi(E))$ alluded to in the Introduction, we have
\begin{align*}
\fd R[x,f] & = \grann_{R[x,f]}(\ann_{\Phi(E)}(\fd R[x,f]))\\ &
\subseteq \grann_{R[x,f]}(\ann_{\Phi(E)}(\fh R[x,f])) = \fh R[x,f]
\\ & \subseteq \grann_{R[x,f]}T = \fd R[x,f].
\end{align*}
Hence $\fh = \fd$ and we have a contradiction.

Thus $\fd$ is a maximal member of the set of proper
$\Phi(E)$-special ideals of $R$; therefore $\fd = \fc$.
\end{proof}

\begin{defi}
[I. M. Aberbach and F. Enescu {\cite[Definition 3.2]{AbeEne05}}]
\label{Splprime} Suppose $(R,\fm)$ is $F$-finite and reduced. Let
$u$ be a generator for the socle $(0:_E\fm)$ of $E$. Aberbach and
Enescu defined
$$
\fP = \left\{r \in R : r\otimes u = 0 \mbox{~in~} R^{(n)}\otimes_RE
\mbox{~for all~}n \gg 0\right\},
$$
an ideal of $R$.
\end{defi}

In \cite[\S 3]{AbeEne05}, Aberbach and Enescu showed that in the
case where $(R,\fm)$ is $F$-finite and $F$-pure, and with the
notation of \ref{Splprime}, the ideal $\fP$ is prime and is equal to
the set of elements $c \in R$ for which, for all $e \in \N$, the
$R$-homomorphism $\phi_{c,e} : R \lra R^{1/p^e}$ for which
$\phi_{c,e}(1) = c^{1/p^e}$ does not split over $R$. Aberbach and
Enescu call this $\fP$ the {\em splitting prime for $R$}. By
\cite[Theorem 4.8(i)]{AbeEne05}, the ring $R/\fP$ is strongly
$F$-regular.

\begin{prop}
\label{ae.1} Suppose that $(R,\fm)$ is $F$-finite and $F$-pure. Let
$\fP$ be Aberbach's and Enescu's splitting prime, as in {\rm
\ref{Splprime}}. Let $\fq$ be the unique largest $\Phi(E)$-special
proper ideal of $R$, as in Theorem {\rm \ref{res.6p}}. Then $\fP =
\fq$.
\end{prop}

\begin{proof} Let $u$ be a generator for the socle $(0:_E\fm)$ of $E$. We can write
$$
\fP = \left\{r \in R : rx^n\otimes u = 0 \mbox{~in~} Rx^n\otimes_RE
\mbox{~for all~}n \gg 0\right\}.
$$
Now for a positive integer $j$ and $r \in R$, if $rx^j \otimes u =
0$ in $\Phi(E)$, then $x(rx^{j-1}\otimes u) = r^px^j\otimes u = 0$,
so that $rx^{j-1}\otimes u = 0$ because the left $R[x,f]$-module
$\Phi(E)$ is $x$-torsion-free. Therefore $$\fP  = \left\{r \in R :
rx^n\otimes u = 0 \mbox{~in~} Rx^n\otimes_RE \mbox{~for all~}n \geq
0\right\}.$$ Let $T$ be the $R[x,f]$-submodule of $\Phi(E)$
generated by $(0:_E\fm) \subseteq R\otimes_RE$. We thus see that $
\fP R[x,f] = \grann_{R[x,f]}T$, and this is $\fq R[x,f]$ by Theorem
\ref{res.6p}. Hence $\fP = \fq$.
\end{proof}

\begin{rmks}
\label{ae.2} Suppose that $(R,\fm)$ is $F$-pure and a homomorphic
image of an excellent regular local ring $S$ of characteristic $p$
modulo an ideal $\fA$. By Theorem \ref{res.6p}(i), there exists a
unique largest $\Phi(E)$-special proper ideal, $\fq$ say, of $R$ and
this is prime. Let $\fQ$ be the unique ideal of $S$ containing $\fA$
for which $\fQ/\fA = \fq$.

\begin{enumerate} \item The results of this section suggest that $\fq$ can be viewed as a
generalization of Aberbach's and Enescu's splitting prime: for
example, Proposition \ref{ae.1} shows that $\fq$ is that splitting
prime in the case where $R$ is, in addition, $F$-finite.

\item Note that $R/\fq$ is strongly $F$-regular (in the sense of
Lyubeznik and Smith mentioned at the beginning of the section).

\item By Proposition \ref{pa.2}, we have $(\fA^{[p^n]} : \fA)
\subseteq (\fQ^{[p^n]} : \fQ)$ for all $n\in\N$. In the special case
in which $S$ is $F$-finite, this result was obtained by Aberbach and
Enescu \cite[Proposition 4.4]{AbeEne05}.
\end{enumerate}
\end{rmks}

\end{document}